\documentclass[12pt]{amsart}
\usepackage{amssymb,latexsym, amsmath, amsxtra,mathrsfs}
\usepackage{bbm}
\usepackage{graphicx,color}
\usepackage{subcaption}
\usepackage{float}
\usepackage{bm}

\newtheorem{conjecture}{Conjecture}
\newtheorem{theorem}{Theorem}

\newtheorem{corollary}{Corollary}
\newtheorem{proposition}{Proposition}
\numberwithin{equation}{section}

\newcommand{\eps}{\varepsilon}

\newcommand{\N}{\mathbb N}
\newcommand{\R}{\mathbb R}
\newcommand{\C}{\mathbb C}

\newcommand{\F}{\mathbb F}
\newcommand{\Sc}{\operatorname{Sc}}

\newcommand{\pr}[1]{\left( #1\right)}

\begin{document}

\title{Averages of long Dirichlet polynomials} 
\author{S. Bettin}
\address{DIMA - Dipartimento di Matematica, Via Dodecaneso, 35, 16146 Genova, Italy}
\email{bettin@dima.unige.it}

\author{J.B. Conrey}
\address{American Institute of Mathematics, 600 East Brokaw Road
San Jose, CA 95112, USA and 
School of Mathematics,
University of Bristol, Bristol, BS8 1TW, United Kingdom
} 
\email{conrey@aimath.org}

\begin{abstract}
We consider the asymptotic behavior of the mean square of truncations of the Dirichlet series of $\zeta(s)^k$. We discuss the connections of this problem with that of the variance of the divisor function in short intervals and in arithmetic progressions, reviewing the recent results on this topic. Finally, we show how these results can all be proved assuming a suitable version of the moments conjecture.
\end{abstract}

\maketitle

 \section{Introduction}
 
 \subsection{Averages of long Dirichlet polynomials}
The moments conjecture for the Riemann zeta-function states that as $T\to\infty$
\begin{equation}\label{moment_conjecture}
\frac 1 T \int_0^T |\zeta(1/2+it)|^{2k}~dt \sim g_k a_k (\log T)^{k^2},\qquad k\in\N,
\end{equation}
where $a_k$ and $g_k$ are the ``arithmetic'' and ``geometric'' constant
\begin{equation*}
a_k := \prod_p \left(1-\frac 1 p\right)^{(k-1)^2}\sum_{j=0}^{k-1} \frac{\binom{k-1}{ j}^2}{p^j},\qquad g_k:=\frac{G(1+k)^2}{G(1+2k)},
\end{equation*}
where $G$ denotes the Barnes $G$-function.
This has been proven for $k=1$ and $k=2$~\cite{HL,Ing} and heuristically obtained for $k=3$ and $k=4$~\cite{CGh,CGo}. The above conjecture for larger (and non integer) values of $k$ was formulated by Keating and Snaith in~\cite{KS} using a random matrix model for the Riemann zeta-function. Other approaches which lead to the same conjecture were later given in~\cite{GH} and in the series of works~\cite{CK1,CK2,CK3,CK4,CK5}, whereas in~\cite{CFKRS} (see also~\cite{CFKRS2}) the conjecture was extended to allow for shifts.
 
The classical approach to~\eqref{moment_conjecture} is that of approximating $\zeta(s)^k=\sum_{n\geq1} d_k(n) n^{-s}$
by appropriate truncations of its Dirichlet polynomial and analysing the mean-square of such. In the pursuit of proving the above conjecture
for values of $k$ larger than 2, one is then lead to consider in general the mean square of Dirichlet polynomials 
with coefficients $d_k(n)$,
$$\mathcal I_k(T,N):=\frac1{T}\int_0^T\bigg| \sum_{1\leq n\leq N} \frac{d_k(n)}{n^{1/2+it}}-\operatorname{Res}_{w=1}\bigg(\frac{\zeta^{k}(w)N^{w-\frac12-it}}{w-\frac12-it}\bigg)\bigg|^2 ~dt,$$
for various values of $k$ and $N$.
In view of~\eqref{moment_conjecture}, one expects $\mathcal I_k(T,T^\alpha)$ to grow on the scale of $(\log T)^{k^2}$ as $T\to\infty$, for any fixed $\alpha>0$. It is thus convenient to define
\begin{eqnarray*}
\mathcal  M_k(\alpha):=\lim_{T\to \infty}\frac{\mathcal I_k(T,T^\alpha)}{a_k  (\log T)^{k^2}},\qquad k\in\N, \alpha>0.
\end{eqnarray*}
If $k=1$, computations analogous to those required for the second moment of $\zeta$ easily give
\begin{eqnarray*}
\frac{1}{T}\int_0^T \bigg|\sum_{n\le N}\frac{1}{n^{1/2+it}}-\frac{N^{\frac12-it}}{\frac12-it}\bigg|^2~dt \sim \left\{ \begin{array}{ll} \log N & \mbox{if $N\le T$,}\\
\log T &\mbox{if $N>T$,}\end{array}\right.
\end{eqnarray*} 
as $N,T\to\infty$. This translates into
\begin{eqnarray*}
\mathcal M_1(\alpha)=  \begin{cases} \alpha & \text{if $0\le \alpha\le 1$,}\\
1 &\text{if $1<\alpha$.}
\end{cases}
\end{eqnarray*} 
Also, standard methods, this time relating to the fourth moment of $\zeta$ give 
\begin{eqnarray}\label{m2conjecture}
\mathcal M_2(\alpha)= \frac1{2^2!}\cdot  \begin{cases} \alpha^4 & 
\text{if $0\le \alpha\le 1$,}\\
 -\alpha^4+8\alpha^3-24\alpha^2+32\alpha -14 &\text{if $1<\alpha\le 2$}\\
2  &\text{if $2<\alpha$.}
\end{cases}
\end{eqnarray} 
(see Proposition~\ref{k2} below). As the asymptotic for the $6$-th moment of $\zeta$ is currently out of reach, one can't hope to prove a similar formula for $\mathcal M_3(\alpha)$ for all $\alpha$. However, 
the  ``recipe'' of~\cite{CFKRS} or heuristic methods such as those of~\cite{CGh} both lead to the conjecture that
\begin{equation}\label{m3conjecture}
\mathcal M_3(\alpha)= \frac1{3^2!}\cdot  \begin{cases} \alpha^9 & 
\mbox{if $0\le \alpha\le 1$}\\[0.2em]
-2 \alpha ^9+27 \alpha ^8-324 \alpha ^7+2268 \alpha ^6-8694 \alpha ^5+{} \\ \quad 19278  
 \alpha ^4-25452 \alpha ^3+19764 \alpha ^2-8343 \alpha +1479 &\mbox{if $1<\alpha\le 2$,}\\[0.2em]
\alpha ^9-27 \alpha ^8+324 \alpha ^7-2268 \alpha ^6+10206 \alpha ^5\\ \quad -30618 \alpha ^4+61236 \alpha ^3-78732 \alpha ^2+59049 \alpha -19641  &\mbox{if $2\le \alpha \le 3$,}\\[0.2em]
42  &\mbox{if $3 \le \alpha$,}
\end{cases}
\end{equation} 

Notice that, as for $\mathcal M_1$ and $\mathcal M_2$,  $\mathcal M_3$ is also predicted to be a continuous piece-wise polynomial. This observation is implicit in the works~\cite{CGo,CK1,CK2,CK3,CK4,CK5}, where the authors heuristically analyse the contributions of the various ranges of the variables of summation  to the moments conjecture main terms. 

The above piece-wise polynomials are also interesting because of their smoothness properties. Indeed, the  graphs of $\mathcal M_2(\alpha)$ and $\mathcal M_3(\alpha)$ given below show that they are smooth, monotonic, and symmetric.  Indeed, the piece-wise polynomial $P$ in~\eqref{m3conjecture} is 8-times continuously differentiable at $\alpha =0 $ and at $\alpha=3$ and is 4-times differentiable at 
$\alpha=1$  and at $\alpha=2$. Also, it satisfies the relation  $P(3-\alpha)=42-P(\alpha)$. In fact, it can be proven that the only piece-wise polynomial $f(\alpha)$, with pieces of degree at most 9, which is 0 for $\alpha<0$ and $\alpha^9$ for $0\leq\alpha< 1$, satisfies $f(3-\alpha)=42-f(\alpha)$ and has the same smoothness properties as $P$ is the piece-wise polynomial given in~\eqref{m3conjecture}.
Indeed, let $f(\alpha)$ share the above properties with $P(\alpha)$. The symmetry and the values of $f(\alpha)$ for $\alpha<1$ determine also the values of $f(\alpha)$ for $\alpha>2$, only leaving the range $1\leq\alpha\leq2 $ in question. Then the symmetry $f(\alpha)+f(3-\alpha)=42$  determines half of the 10 coefficients of $f$ in this interval and the 4-times smoothness at $\alpha=1$ determines the other 5.  

For larger values of $k$, it easily follows from~\cite[Corollary~3]{MV} that $\mathcal M_k(\alpha)=\frac1{k^2!}\alpha^{k^2}$ for $\alpha<1$, and one still expects $\mathcal M_k(k-\alpha)=g_k-\mathcal M_k(\alpha)$ as this is suggested by the moment conjecture~\eqref{moment_conjecture} and the functional equation for $\zeta$. We expect that also the piece-wise structure and the smoothness properties generalise to $k>3$. We refer to Conjecture~\ref{smoothnes_conjecture} below for a formal statement of these properties.

\begin{figure}[hh]
  \begin{subfigure}[t]{0.475\textwidth}
    \includegraphics[width=\textwidth]{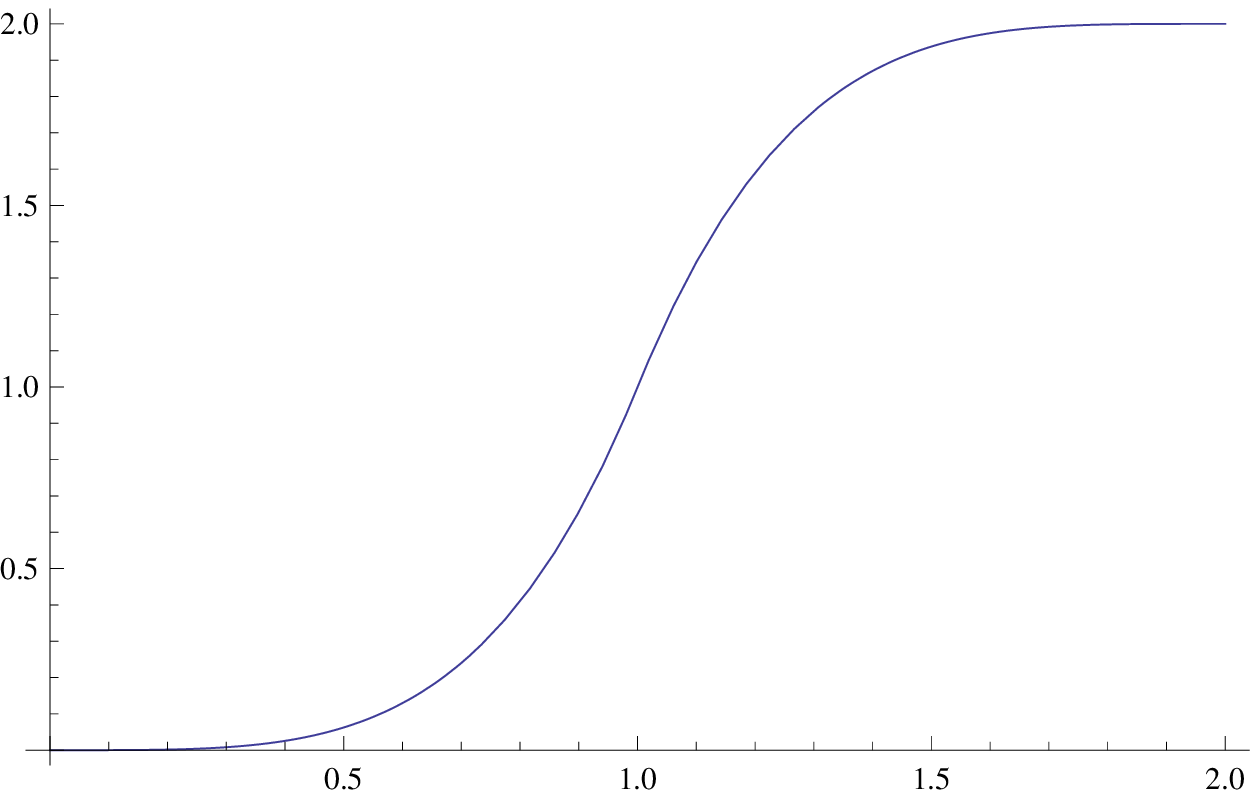}
    \caption{\sf The plot of $4!\mathcal M_2(\alpha)$, as conjectured in~\eqref{m2conjecture}, for $0<\alpha<2$.}
    \label{fig-a}
  \end{subfigure}\hfill
  \begin{subfigure}[t]{0.475\textwidth}
    \includegraphics[width=\textwidth]{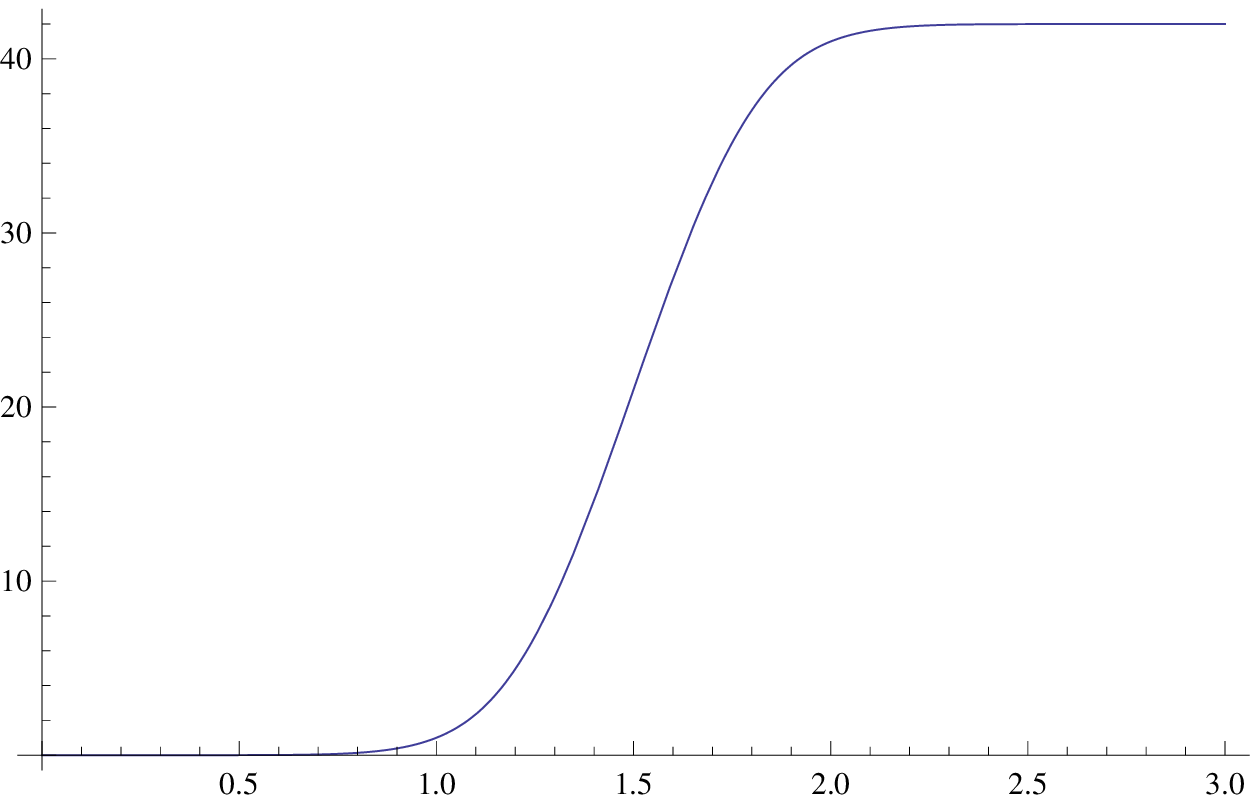}
    \caption{\sf The plot of $9!\mathcal M_3(\alpha)$, as conjectured in~\eqref{m3conjecture}, 
    for $0<\alpha<3$.}
    \label{fig-b}
  \end{subfigure}
  \label{fig:main}
\end{figure}

For $0\leq r\leq k$, let
$$
P_{r,k}(\alpha):=\frac1{(2\pi i)^kk!}\int_{C_1(r)}\hspace{-0,5em}\cdots \int_{C_k(r)}\bigg(\sum_{j=1}^kw_j+\alpha\bigg)^{k^2}\bigg(\prod_{i\neq j}(w_i-w_j)\bigg)\bigg(\prod_{i=1}^kw_i^{-k}(w_i+1)^{-k} \bigg)\boldsymbol {dw}
$$
where $\boldsymbol {dw}:=dw_1\cdots dw_r$ and where $C_i(r)$ denotes the circle, oriented counter-clockwise, with radius $\frac12$ and center $-1$ if $i=1,\dots, r$ and $0$ otherwise. 
Notice that $P_{r,k}(\alpha)$ is a polynomial of degree $k^2$. Also, it has a zero of order at least $(k-r)^2 + r^2$ at $\alpha=r$ (see Proposition~\ref{p1} below).  If $r>k$ let $P_{r,k}:=0$.  

\begin{conjecture}\label{smoothnes_conjecture}
Let $k\in\N$. For $\alpha\in\R$ one has 
\begin{equation}\label{formk}
\mathcal M_k(\alpha)=\frac{1}{k^2!}\sum_{0\leq r<\alpha}\binom{k}r P_{r,k}(\alpha).
\end{equation}
In particular, $\mathcal M_k(\alpha)$ is continuously differentiable  $ k^2 - 2 k \ell + 2 \ell^2-1$ times at $\alpha=\ell$ for all $\ell=0,\dots,k$.
\end{conjecture}

For $k\geq4$, the smoothness conditions in Conjecture~\ref{smoothnes_conjecture}, the symmetry and the values for $\alpha<1$ are not sufficient to determine $\mathcal M_k(\alpha)$. It would be interesting to see if one can characterise $\mathcal M_k(\alpha)$ by adding some other suitable simple conditions. 

We will show that Conjecture~\ref{smoothnes_conjecture} can be proven assuming a version of the moments conjecture~\eqref{moment_conjecture} in the simplified setting where all the $\zeta(1/2+it)$ and all $\zeta(1/2-it)$ are shifted by the same quantities and only the leading term is kept (see~\cite{CFKRS} and~\cite[Corollary~1]{CRS}). Let $k\in\N$ and let $Y=Y(T)$ be a parameter to be specified in the applications of the Conjecture.  We believe that $Y= T^A$ for any fixed $A\in\R$ should be admissible.

\begin{conjecture}[Shifted moments conjecture]\label{shifted moments conjecture}
For $T\geq1$ one has 
\begin{equation}\label{smce}
\begin{split}
&\frac1T\int_{0}^T\zeta(\tfrac12+it+s-z)^{k}\zeta(\tfrac12-it+z)^{k}\,dt\\[-0.5em]
&\hspace{7em}= \frac{a_k}{k!(2\pi i)^k}\iint_{R^k}T^{\sum_{j=1}^kw_j}\prod_{i=1}^kw_i^{-k}(w_i+s)^{-k}\prod_{i\neq j}(w_i-w_j)  \boldsymbol {dw} \\
& \hspace{7em}\quad+O_k\big((\log T)^{k^2-1+\eps}(1+\min(|s|,1)\log T)^{-k^2/2+1}\big)
\end{split}
\end{equation}
for $|\Re(z)|,|\Re(s)|\leq3/\log T$, $|\Im (z)|,|\Im (s)|\ll Y
$ and $R$ any piece-wise smooth path encircling once $0$ and $-s$  in the positive direction.
\end{conjecture}
The conjecture in a stronger form, including lower order terms, is known for $k=1$ and $Y=T^{2-\eps}$ (see~\cite{Bet}), whereas for $k=2$ and $Y=T^\xi$, for some small $\xi>0$, the proof of the Conjecture is implicit in~\cite{HY,BBLR} (see also~\cite{Shi}).
\begin{theorem}\label{t1}
Let $k\geq1$. Assume Conjecture~\ref{shifted moments conjecture} for $Y=T^{k/2+\eps}$ for some $\eps>0$, then Conjecture~\ref{smoothnes_conjecture} holds.
\end{theorem}
We notice that the assumption of the uniformity with respect to the shifts can be reduced to $Y\ll (\log T)^\eta$ for any fixed $\eta>0$ if one considers the analogous problem where the sum over $n$ is truncated by a smooth cut-off.

\subsection{The divisor function in short intervals and in arithmetic progressions}

A phenomenon similar to that described for $\mathcal M_k(\alpha)$ arises also when analysing the variance of the divisor function over short intervals. 
Given, $x,H>0$, let 
\begin{align*}
\Delta_k(x,H):=\sum_{x<n\leq x+H}d_k(n)-\operatorname{Res}_{s=1}\Big(\zeta^{k}(s)\frac{(x+H)^s-x^s}s\Big)
\end{align*} 
be the error term in the $k$-th divisor problem over the interval $(x,x+H]$ and let
\begin{align*}
\mathcal V_k(X,H):=\frac1{HX}\int_{X}^{2X}|\Delta_k(x,H)|^2\,dx
\end{align*} 
be the variance of $\Delta_k(x,H)/\sqrt H$ on average over $x\asymp X$. The asymptotic for $\mathcal V_k(X,H)$ (or for averages very close to it) in the case $H=X^{1-1/\alpha}$, $\alpha>0$, was computed in~\cite{Jut,CSa,Ivi,Ivi2} for $k=2$ and by Lester~\cite{Les} for $\alpha>k-1$ if $k=3$ and, assuming the Lindel\"of hypothesis, if $k\geq4$. We also refer to~\cite{MT} for some upper bounds.

The analogous problem in function fields was solved by Keating, Rodgers, Roditty-Gershon, and Rudnick~\cite{KRRR}.
Denoting by $\mathscr M_n$ the set of degree $n$ monic polynomials in $\F_q[x]$ and by $d_k(f)$ the number of decompositions of a monic $f\in\F_q[x]$ as $f = f_1f_2\cdots f_k$ with $f_i\in\F_q[x]$ monic, they showed that, for $h\geq0$, $n\geq3k$\footnote{
In~\cite{KRRR} this condition is replaced by $0\leq h\leq \min(n-5,n-n/k-2)$ which, for $n\geq 3k$, rewrites as $h\leq n-n/k-2$. This condition can then be dropped since it trivially implies that $I_k(n,n-h-2)=0$, whereas the left hand side of~\eqref{KR3} is $O(q^{h+\frac12})$ by~\cite[(1.30-31)]{KRRR}.}, one has
 \begin{align}\label{KR3}
\frac1{q^n}\sum_{A\in\mathscr M_n}\bigg|\sum_{f\in M_n\atop
\deg(f-A)\leq h}d_k(f)-q^{h+1}\binom{n+k-1}{k-1}\bigg|^2= q^{h+1}I_k(n,n-h-2)+O(q^{h+\frac12})
\end{align}
 as $q\to\infty$,  
where $I_k(n,N)$ is the matrix integral
\begin{align*}
I_k(n,N):=\int_{U(N)}\bigg|\sum_{j_1+\cdots+j_k=n\atop 0\leq j_1,\dots,j_k\leq N}\Sc_{j_1}(A)\cdots \Sc_{j_k}(A)\bigg|^2\,dA,
\end{align*}
with $dA$ denoting the Haar measure on the unitary group $U(N)$ and $\Sc_j(A)$ denoting the $j$-th secular coefficient of $A$, i.e.
\begin{align*}
\Lambda_A(z):=\det(\mathbbm 1_n+zA):=\sum_{j=0}^n\Sc_j(A)z^j.
\end{align*}
They then show that for all $m,N\in\N$ one has
\begin{align}\label{ffik}
I_k(n,N)=\gamma_k(n/N)N^{k^2-1}+O_k(N^{k^2-2}),
\end{align}
where 
\begin{align*}
\gamma_k(\alpha):=\frac1{k! G(1+k)^2}\iint_{[0,1]^k}\delta(w_1+\cdots +w_k-\alpha)\bigg(\prod_{i<j}(w_i-w_j)^2\bigg)\,\boldsymbol {dw}
\end{align*}
and $\delta$ is the Dirac $\delta$-function. Other expressions of $\gamma_c$ are also known. Indeed, in~\cite{KRRR} $\gamma_c$ is expressed in terms of a lattice
point count, whereas in~\cite{BGR} it is shown that $\gamma_k$ is the inverse Fourier transform of a Hankel determinant satisfying a Painlev\'e V equation. We also refer to~\cite{GR}, where the case of non-integer $k$ is considered and where a connection with $z$-measures is highlighted, and to Section~\ref{rmt} below for an alternative complex analytic approach to~\eqref{ffik}.

The usual analogy between the number field and the function field case can then be used to translate this result to the number field setting. For $\alpha>0$ and $k\in\N$, let
\begin{eqnarray*} 
\mathcal C_k(\alpha):=\lim_{X\to \infty}\frac{ \mathcal V_k(X,X^{1- 1/\alpha})}{a_k  (\log (X^{1/\alpha}))^{k^2-1}}.
\end{eqnarray*}
Then, one expects the following to hold (see~\cite[Conjecture~1.1]{KRRR}).
\begin{conjecture}\label{conjKR3}
For $k\geq1,\alpha>0$ we have $\mathcal C_k(\alpha)=\gamma_k(\alpha)$.
\end{conjecture}
The function  $\gamma_k(\alpha)$ is easily seen to be a piece-wise polynomial of degree $k^2-1$, which is supported  on $[0,k]$ and satisfies $\gamma_k(\alpha)=\gamma_k(k-\alpha)$. Also, in~\cite{KRRR} it is shown that $\gamma_k(\alpha)$ can be written as 
$$
\gamma_k(\alpha)=\sum_{0\leq r<\alpha}Q_{r,k}(\alpha),
$$
for certain polynomials $Q_{r,k}$ of degree $k^2-1$ with a zero of order at least $(k-r)^2+r^2-1$ at $\alpha=r$. In particular $\gamma_k(c)$ is differentiable at least $(k-r)^2+r^2-2$ times at $\alpha=r$ (see also~\cite{BGR} for an alternative proof). 

The similarity between the properties of $\mathcal M_k$ and $\mathcal C_k$ suggests there might be a relation between them. Indeed, the comparison of $\mathcal M_k$ and $\mathcal C_k$ for the first few values of $k$ led Basor, Ge and Rubinstein to conjecture that $M'_k(\alpha)=\mathcal C_k(\alpha)$.\footnote{This also explains the observation that $\int_0^k\gamma_k(c)=g_k$ made in~\cite[Footnote 2]{RS}} Under the shifted moments conjecture, we show that this is indeed the case.

\begin{theorem}\label{scd}
Let $k\geq1$. Assume Conjecture~\ref{shifted moments conjecture} for $Y\ll (\log T)^\eta$ for some fixed $\eta>0$. Then for all $\alpha>0$ we have  $\mathcal C_k(\alpha)=\frac{1}{k^2!}\sum_{0\leq r<\alpha}\binom{k}r P'_{r,k}(\alpha)=\gamma_k(\alpha)$. In particular, Conjecture~\ref{conjKR3} holds.
\end{theorem}

Notice that since Conjecture~\ref{shifted moments conjecture} is known for $k\leq2$ with the desired uniformity, then the above theorem gives also an alternative unconditional proof of the case $k=2$.

 \medskip
 
 Much of what described above can also be similarly stated for the divisor function in arithmetic progressions. In this case, one defines 
\begin{align*}
\mathcal W_k(X;q):=\frac1X\sum_{1\leq a\leq q\atop (a,q)=1}\bigg|\sum_{n\leq X\atop n\equiv a\, (\textnormal{mod }q)}d_k(n)-\frac1{\varphi(q)}\operatorname{Res}_{s=1}\Big(\zeta_q^{k}(s)\frac{X^s}s\Big)\bigg|^2, \qquad q\in\N, x>0,
\end{align*} 
where $\varphi$ is Euler's $\varphi$-function and $\zeta_q(s):=\sum_{n\geq1, (n,q)=1}d_k(n)n^{-s}$. Then, writing $a_k(q):=\lim_{s\to1}(s-1)^{k^2}\sum_{(n,q)=1}\frac{d_k(n)^2}{n^s}$ (notice that $a_k(q)\to a_k$ if $q\to\infty$ among primes), we define 
 \begin{eqnarray*} 
\mathcal D_k(\alpha):=\lim_{q\to \infty,\ q\textnormal{ prime}\atop X\asymp q^{\alpha}}\frac{\mathcal W_k(X,q)}{a_k(q)  (\log q)^{k^2-1}},\qquad \alpha>0,
\end{eqnarray*}
where we added the assumption of the primality of $q$ for simplicity.

In~\cite{KRRR} the authors consider the function field analogue of this problem and are lead to conjecture that $\mathcal D_k(\alpha)$ also coincides with $\gamma_k(\alpha)$.

\begin{conjecture}\label{conjKR3_2}
For $k\geq1,\alpha>0$ we have $\mathcal D_k(\alpha)=\gamma_k(\alpha)$.
\end{conjecture}

The case of $k=2$ was obtained in~\cite{LZ} (see also~\cite{Mot,Blo}) and in general partial results are known if $\alpha>k-\frac12$~\cite{KR}. Recently Rodgers and Soundararajan~\cite{RS} proved a smoothed version of Conjecture~\ref{conjKR3_2} when one introduces also an average over $q\asymp Q$ provided that $0<\alpha<1+2/k$. Under the Generalized Lindel\"of Hypothesis they can also extend their result to $0<\alpha<2$ (see also~\cite{HS} for an unconditional lower bound on this range).
See also~\cite{dF}, where the authors consider a sum similar to $\mathcal W_k(X;q)$, in which the major arc contribution, rather than the main term,  is subtracted.

Also in this case Conjecture~\ref{conjKR3_2} can be proved under a moments conjecture, with the difference that now the moments needed are those of Dirichlet $L$-functions.

\begin{conjecture}\label{shifted moments conjecture dirichlet}
Let $q$ be prime. Then as $q\to\infty$ one has 
\begin{equation}\label{smce}
\begin{split}
&\frac1{\varphi(q)}\sum_{\chi \mod q}L(\tfrac12+s-z,\chi)^{k}L(\tfrac12+z,\overline \chi)^{k}\\[-0.5em]
&\hspace{7em}= \frac{a_k}{k!(2\pi i)^k}\iint_{R^k}q^{\sum_{j=1}^kw_j}\prod_{i=1}^kw_i^{-k}(w_i+s)^{-k}\prod_{i\neq j}(w_i-w_j)  \boldsymbol {dw} \\
& \hspace{7em}\quad+O_k\big((\log q)^{k^2-1+\eps}(1+\min(|s|,1)\log q)^{-k^2/2+1}\big)
\end{split}
\end{equation}
for $|\Re(z)|,|\Re(s)|\leq3/\log q$, $|\Im (z)|,|\Im (s)|\ll Y
$ and $R$ any piece-wise smooth path encircling once $0$ and $-s$  in the positive direction.
\end{conjecture}

Notice that this conjecture is a weaker form of~\cite[Conjecture]{CIS}, with only the leading term being kept.

\begin{theorem}\label{tt3}
Let $k\geq1$. Assume Conjecture~\ref{shifted moments conjecture dirichlet} for $Y=q^{k+\eps}$ for some $\eps>0$. Then, Conjecture~\ref{conjKR3_2} holds.
\end{theorem}

Also in this case the uniformity with respect to the shift can be relaxed to $O((\log q)^\eta)$ for some fixed $\eta>0$ if one considers the analogue of $\mathcal W_k(X;q)$ with a smooth cut-off.


Conjecture~\ref{conjKR3_2} have been generalized to arithmetic functions associated to other $L$-functions in~\cite{HKR}. Also, in~\cite{CR} a variant of the (shifted) quadratic analogue of Conjecture~\ref{formk} is considered. 

We conclude the introduction by mentioning that the use of the moments conjecture in Theorems~\ref{t1},~\ref{scd} and~\ref{tt3} is reminiscent of the work~\cite{CS} where the authors use the ratio conjectures to deduce several results in number theory. Also, the study of problems analogous to  those treated here, with $d_k(n)$ replaced by the von Mangoldt function $\Lambda(n)$, has a rich history. See for example~\cite{GGM} and references therein.

\section*{Acknowledgment}
The authors wish to thank Sary Drappeau, Brad Rodgers and the anonymous referee for useful comments.
S. Bettin is member of the INdAM group GNAMPA and his work is partially supported by PRIN 2017 ``Geometric, algebraic and analytic methods in arithmetic''. JBC is supported in part by a grant from the NSF.

\section{The proofs of the theorems}
We start by showing that Conjecture~\ref{smoothnes_conjecture} holds for $k=2$.
\begin{proposition}\label{k2}
Conjecture~\ref{smoothnes_conjecture} holds for $k=2$.
\end{proposition}
\begin{proof}
It suffices to show that~\eqref{m2conjecture} holds. Also, we can assume $T$ is a half-integer.

The case of $\alpha\leq1$ follows from~\cite[Corollary~3]{MV}, whereas the case $\alpha\geq2$ can be deduced from~\cite[(4.18.2)]{Tit} (with $y=O(1)$) and the fourth moment for the Riemann zeta-function. Finally, if $1<\alpha<2$, then the approximate functional equation~\cite[(4.18.2)]{Tit} for $\zeta^2$ implies $\mathcal M_2(2-\alpha)+\mathcal M_2(\alpha)+E=g_2$, where
\begin{align*}
E=\lim_{T\to \infty}\frac{2}{a_2  T(\log T)^{4}}\Re\int_0^T\chi(\tfrac12-it)^2 \sum_{1\leq n\leq T^\alpha} \frac{d(n)}{n^{1/2+it}}\sum_{1\leq m\leq t^2/(4\pi^2T^\alpha)} \frac{d(m)}{m^{1/2+it}}\,dt.
\end{align*}
The result then folloes since, similarly to~\cite[p. 143]{Tit}, one has that the inner integral is 
\begin{align*}
\ll \sum_{1\leq n\leq T^\alpha} \frac{d(n)}{n^{1/2}}\sum_{1\leq m\leq T^2/(4\pi^2T^\alpha)} \frac{d(m)}{m^{1/2}}\bigg(1+O\bigg(\frac1{\sqrt {nm}|\log(4\pi^2nm /T)|}\bigg)\bigg)\ll T(\log T)^2.
\end{align*}

\end{proof}

Next, we observe that Conjecture~\ref{shifted moments conjecture} implies the following bound.
\begin{corollary}\label{corcon}
Assume Conjecture~\ref{shifted moments conjecture}. Then, under the same hypothesis of the conjecture we have
\begin{align*}
\frac1T\int_{0}^T\zeta(\tfrac12+it+s-z)^{k}\zeta(\tfrac12-it+z)^{k}\,dt
\ll(\log T)^{k^2}(1+\min(|s|,1)\log T)^{-k^2/2},
\end{align*}
and
\begin{align*}
\frac1T\int_{0}^T|\zeta(\tfrac12+it)|^{2k}\,dt
\ll(\log T)^{k^2},
\end{align*}
where here and throughout this section the error terms are allowed to depend on $k$.
\end{corollary}
\begin{proof}
If $|s|< 3/\log T$ it suffices to take the path $R$ in Conjecture~\ref{shifted moments conjecture} to be a circle centred at zero of radius $4/\log T$ and bound trivially the integrals on the right of~\eqref{smce}. If $|s|\geq 3/\log T$, by Cauchy's theorem we can write these integrals as a sum of $2^k$ integrals where each $w_i$ is integrated along a circle of radius $1/\log T$ and center either $0$ or $-s$. The contribution of the term where the center at $0$ is taken exactly $0\leq r\leq k$ times is 
$$\ll (\log T)^{k^2-k-r(r-1)-(k-r)(k-r-1)} \cdot |s|^{-k^2+2r(k-r)}.$$
The result then follows since the first factor has maximum $(\log T)^{k^2/2}$ whereas the second factor is $\ll \max(1,|s|^{-k^2/2})$. 
\end{proof}

\subsection{Proof of Theorem~\ref{t1}}
By Proposition~\ref{k2} we can assume $k\geq3$. If $\alpha>k$, one can directly relate $\mathcal I_k(T,T^{\alpha})$ with the $2k$-th moment of the Riemann zeta-function, and so the desired formula follows by Conjecture~\ref{shifted moments conjecture}. In particular we can assume $0<\alpha\leq k$. 
We assume Conjecture~\ref{shifted moments conjecture} holds with $Y=T':=T^{k/2+\eps}$ for some fixed $\eps>0$. Let $\eps'\leq \eps/3$ be a sufficiently small positive constant so that $3\eps'<\alpha\leq k$. Also, we assume $\alpha$ is not an integer and we write $N:=[T^\alpha]+\frac12$.

By~\cite[Lemma 3.12]{Tit} we have
\begin{align*}
\sum_{1\leq n\leq N} \frac{d_k(n)}{n^{1/2+it}}&=\frac1{2\pi i}\int_{\frac12+c-i T'}^{\frac12+c+iT'}\zeta(\tfrac12+it +z)^k \frac{N^{z}}{z}\,dz+O(T^{-\eps'})\\
&=\frac1{2\pi i}\int_{c-i T'}^{c+iT'}\zeta(\tfrac12+it +z)^k \frac{N^{z}}{z}\,dz+\operatorname{Res}_{z=1}\bigg(\frac{\zeta^{k}(z)N^{z-\frac12-it}}{z-\frac12-it}\bigg)+O(T^{-\eps'}),
 \end{align*}
 with $c=1/\log T$,
 since Corollary~\ref{corcon} implies in particular that $\zeta(\frac12+it)\ll_{k} 1+|t|^{(1+\eps')/(2k)+\eps'}$. By Cauchy-Schwartz, it  follows that 
 $$\mathcal I_k(T,N)= \mathcal I+O(\sqrt{\mathcal I}T^{-\eps'}+ T^{-2\eps'}),$$
where
\begin{align*}
\mathcal I&:=\frac1{4\pi^2T} \int_0^T\bigg| \int_{c-i T'}^{c+iT'}\zeta(\tfrac12+it +z)^k \frac{N^{z}}{z}\,dz\bigg|^2\,dt\\
&=\frac1{4\pi^2} \int_{c-i T'}^{c+iT'}\int_{-c-i T'}^{-c+iT'}\frac1T \int_0^T  \zeta(\tfrac12+it +z_1)^k \zeta(\tfrac12-it -z_2)^k\,dt\frac{N^{z_1-z_2}}{z_1z_2}\,dz_2 dz.
\end{align*}
We evaluate the inner integral using Conjecture~\ref{shifted moments conjecture}. With a quick computation we see that the contribution of the error term is 
\begin{align*}
&\ll_k (\log T)^{k^2-1} \int_{- T'}^{T'}\int_{- T'}^{T'} (1+\min(|z_2-z_1|,1)\log T)^{-k^2/2+1} \frac{d z_2 dz_1}{|c+iz_1| | c+iz_2|}\\
&\ll_k (\log T)^{k^2-1},
\end{align*}
provided that $k\geq3$. 
For the main term, we choose the path $R$ to be the border of the rectangle  $[-\frac{3}{\log T},\frac{3}{\log T}]+i[-3T,3T]$ (in the positive direction). Exchanging the order of integration, we obtain that the contribution of the main term is 
\begin{align}
\frac{a_k}{(2\pi i)^{k}}   \iint_{R^n}&T^{\sum_{j=1}^kw_j} \prod_{i=1}^kw_i^{-k}  \prod_{i\neq j}(w_i-w_j)\notag\\
&\times \bigg( \frac{-1}{(2\pi i)^2}\int^{c+iT'}_{c-iT'}\int^{-c+iT'}_{-c+iT'} \prod_{i=1}^k(w_i+z_1-z_2)^{-k}
\frac{N^{z_1-z_2}}{z_1z_2}dz_2 dz_1\bigg)\boldsymbol {dw}.\label{mainterm1}
\end{align}
We now consider the two inner integrals. We have
\begin{align*}
&\frac{-1}{(2\pi i )^2}\int^{c+iT'}_{c-iT'}\int^{-c+iT'}_{-c+iT'} \prod_{i=1}^k(w_i+z_1-z_2)^{-k}
\frac{N^{z_1-z_2}}{z_1z_2}dz_2 dz_1\\
&\hspace{6em}=\frac{-1}{2\pi i }\int^{2c+2iT'}_{2c-2iT'}N^{z}\prod_{i=1}^k(w_i+z)^{-k} \frac1{2\pi i }\int_{-c-iT'+i\max (-z,0)}^{-c+iT'+i\min(-z,0)} 
\frac{1}{(z+z_2)z_2}dz_2 dz\\
&\hspace{6em}=\frac1{2\pi i }\int^{2c+2iT'}_{2c-2iT'}N^{z}\bigg(\frac1{z}+O(T'^{-1}\log T)\bigg)\prod_{i=1}^k(w_i+z)^{-k}  dz_1,
\end{align*}
where for the last equality we used that the inner integral is $-\frac1{z}+O(T'^{-1} \log T)$. This can be proved by extending and evaluating the integral if $|\Im(z)|<T'/2$, whereas for  $|\Im(z)|>T'/2$ one has that the integral is $O(T'^{-1}\log T)=-\frac1{z}+O(T'^{-1} \log T)$ by a trivial bound. We insert this approximation in~\eqref{mainterm1}. The contribution of the error term can then be bounded by $O(T'^{-1+\eps})$, and so we obtain that, up to a $O( (\log T)^{k^2-1})$ error, $\mathcal I$ is
\begin{align*}
\frac{a_k}{(2\pi i)^k}   \iint_{R^k}T^{\sum_{j=1}^kw_j}\prod_{i\neq j}(w_i-w_j)\bigg(\frac1{2\pi i}  \int^{2c+iT'}_{2c-iT'} \prod_{i=1}^kw_i^{-k}(w_i+z)^{-k}
\frac{N^{z}}{z} dz\bigg) \boldsymbol {dw}.
\end{align*}
Since $|z|\geq2/\log T$, by Cauchy's theorem we can split each outer integral as a sum of two integrals along two circular paths (in the positive direction) of radii $1/{\log T}$ and centers $-\delta_iz$, with $\delta_i=0,1$. Therefore, 
\begin{align*}
\mathcal I&=a_k\sum_{\bm \delta\in\{0,1\}^k}\mathcal {I}_{\bm{\delta}}
+O\big( (\log T)^{k^2-1}\big),
\end{align*}
where
\begin{align*}
\mathcal {I}_{\bm{\delta}}:=\frac{1}{(2\pi i)^k}   \iint_{|w_i|=\frac1{\log T}} \bigg(\frac1{2\pi i}\int^{2c+iT'}_{2c-iT'}&T^{\sum_{j=1}^k(w_j-\delta_jz)} \prod_{i=1}^k(w_i-z\delta_i)^{-k}(w_i-z\delta_i+z)^{-k}\\
&\times 
 \prod_{i\neq j}(w_i-w_j-\delta_i z+\delta_j z)\,\frac{N^{z}}{z} dz\bigg)\boldsymbol {dw}.
\end{align*}
If $\alpha-\sum_{j=1}^{k}\delta_j<0$ (and $T$ is sufficiently large) we have $\mathcal {I}_{\bm{\delta}}=O(T'^{-1+\eps})$ as can be seen by moving the line of integration of the inner integral to $+\infty$ bounding trivially the contribution of the horizontal lines. In particular, 
\begin{align*}
\mathcal I&=a_k\sum_{\bm \delta\in\{0,1\}^k\atop \sum_{j=1}^{k}\delta_j<\alpha}\mathcal {I}_{\bm{\delta}}
+O\big( (\log T)^{k^2-1}\big).
\end{align*}
If $\alpha-\sum_{j=1}^{k}\delta_j>0$ then we move the line of integration in $\mathcal {I}_{\bm{\delta}}$ to $-\infty$ and see that the inner integral is the sum of the residue at $z=0$, $z=w_i$ or $z=-w_i$ for each $i$, up to a $O(T'^{-1+\eps})$ error term. Thus, in the latter case, up to a small error, we can replace the inner integral by an integral over the circle $|z|=2/\log T$. For  $\alpha-\sum_{j=1}^{k}\delta_j>0$  we then obtain
\begin{align}
\mathcal {I}_{\bm{\delta}}=\frac{1}{(2\pi i)^k} \iint_{|w_i|=\frac1{\log T}}\bigg(\frac1{2\pi i}\int_{|z|=\frac2{\log T}}  & T^{\sum_{j=1}^k(w_j-\delta_jz)}  \prod_{i=1}^k(w_i-z\delta_i)^{-k}(w_i-z\delta_i+z)^{-k}\notag \\
&\times
 \prod_{i\neq j}(w_i-w_j-\delta_i z+\delta_j z)\frac{N^{z}}{z}\, dz\bigg)\boldsymbol {dw}+O(T'^{-1+\eps}).\label{ifac}
\end{align}
Making the change of variables $w_i\to zw_i$ for each $i$ (exchanging twice the order of integration),
we then see that this is
\begin{align*}
\frac{1}{(2\pi i)^{k}} \iint_{|w_i|=\frac12} &\bigg(\frac1{2\pi i}\int_{|z|=\frac2{\log T}}   T^{z (\alpha+\sum_{j=1}^k(w_j-\delta_j))} \,\frac{dz}{z^{1+k^2}} \bigg) \\
&\times \prod_{i=1}^k(w_i-\delta_i)^{-k}(w_i-\delta_i+1)^{-k}
\prod_{i\neq j}(w_i-w_j-\delta_i +\delta_j )\boldsymbol {dw}+O(T'^{-1+\eps}).
\end{align*}
By the residue theorem the result then follows. If $\alpha$ is an integer, it suffices to repeat the same argument with $N=T^{\alpha-\frac1T}$.

\subsection{Proof of Theorem~\ref{scd}}
Let $H=X^{1-1/\alpha}$ with $\alpha>0$. We assume $k\geq2$ otherwise the result is trivial. We shall prove that $\mathcal C_k(\alpha)=\frac{1}{k^2!}\sum_{0\leq r<\alpha}\binom{k}r P'_{r,k}(\alpha)$. By Corollary~\ref{cff} below it then follows that this is  also equal to $\gamma_k(\alpha)$.

Also, writing
\begin{align*}
\tilde V_{g,k}(X,H):=\frac1{HX}\int_{\R}g(x/X)|\Delta_k(x,H)|^2\,dx,
\end{align*} 
by a simple approximation argument on has that it suffices to prove that
\begin{eqnarray*} 
\lim_{X\to \infty}\frac{ \tilde V_{g,k}(X,X^{1- 1/\alpha})}{a_k  (\log X^{1/\alpha})^{k^2-1}}=\int_\R g(x)\,dx\cdot\frac{1}{k^2!}\sum_{0\leq r<\alpha}\binom{k}r P'_{r,k}(\alpha)
\end{eqnarray*}
for all smooth functions $g$ of compact support contained in $[\frac12,4]$.

Let $T=X^{1+\eps}/\sqrt H$. Then by~\cite[Lemma 3.12]{Tit} we obtain 
\begin{align*}
\mathcal V_k(X,H)&=\frac1{HX}\int_{X}^{2X}g(x/X)\bigg|\frac1{2\pi  }\int_{-T}^{T} \zeta^k(\tfrac12+iz)  \frac{(x+H)^{\frac12+iz}-x^{\frac12+iz}}{{\frac12+iz}}\,dz \bigg|^2\,dx+o(\log X)^{k^2-1}).
\end{align*}
We shall now show that the contribution from the $|z|\notin [\frac{X}{HZ}, \frac{XZ}{H}]$ is negligible, that is 
\begin{align}\label{bame}
\frac1{HX}\int_{X}^{2X}g(x/X)\bigg|\frac1{2\pi  }\int_{|z|\in [0, \frac{X}{HZ}]\cup[\frac{XZ}{H}, T]} \zeta^k(\tfrac12+iz)  \frac{(x+H)^{\frac12+iz}-x^{\frac12+iz}}{{\frac12+iz}}\,dz \bigg|^2\,dx=o(\log X)^{k^2-1})
\end{align}
where $Z>1$ is a parameter, to be specified later, which goes to infinity slowly. Proceeding as in~\cite[Page 1036]{MR1} (see also~\cite[Page 5]{MR2} and~\cite[Page 25]{SV}), we see that~\eqref{bame} follows if we can prove that
\begin{align*}
&\frac1{H}\int_\R g(x/X)\bigg|\frac1{2\pi i }\int_{|z|\in [0, \frac{X}{HZ}]\cup[\frac{XZ}{H}, T]} \zeta(\tfrac12+iz)^k x^{iz} \frac{(1+u)^{\frac12+iz}-1}{{\frac12+iz}}\,dz \bigg|^2\,dx=o(\log X)^{k^2-1})
\end{align*}
for all $u\ll H/X$. We divide the range of integration over $z$ as $E_1\cup E_2\cup E_3$ with $E_1:=[0,\frac{X}{HZ}]$,  $E_2:=[\frac{XZ}{H},\frac{XZ}{H}\log X]$ , $E_3:=[\frac{XZ}{H} \log X,T]$.

Expanding the square and denoting by $\hat g$ the Mellin transform of $g$, this is
\begin{align*}
&\frac X{H}\frac1{(2\pi)^2 }\iint_{|z_1|,|z_2| \in E_1\cup E_2\cup E_3} X^{iz_1-iz_2}  \hat g(1+iz_1-iz_2)\zeta(\tfrac12+iz_1)^k\zeta(\tfrac12-iz_2)^k\\
&\hspace{10em}\times \frac{(1+u)^{\frac12+iz_1}-1}{{\frac12+iz_1}}\frac{(1+u)^{\frac12+iz_2}-1}{{\frac12+iz_2}}\,dz_1dz_2\\
&=\frac X{H}\frac1{(2\pi)^2 }\iint_{|z|,|z-s| \in E_1\cup E_2\cup E_3}  X^{is}\hat g(1+is)\zeta(\tfrac12+iz)^k\zeta(\tfrac12-iz+is)^k\frac{F(z,s,u)}{(\frac12+iz)(\frac12-iz+is)}\,dzds,
\end{align*}
where
\begin{align*}
F(z,s,u)&:=\big((1+u)^{\frac12+iz}-1\big)\big((1+u)^{\frac12-iz+is}-1\big)\\
&=(1+u)^{1+is}+1-(1+u)^{\frac12-iz+is}-(1+u)^{\frac12+iz}.
\end{align*}
Now,
\begin{align*}
F(z,s,u)&\ll \min(|\tfrac12+iz| H/X ,1)\min(|\tfrac12-iz+is| H/X,1).\\
\end{align*}
Thus, using
\begin{align*}
\frac{\zeta(\tfrac12+iz)^k\zeta(\tfrac12-iz+is)^kF(z,s,u)}{(\frac12+iz)(\frac12-iz+is)}&\ll |\zeta(\tfrac12+iz)|^{2k}\min( (H/X)^2 ,|\tfrac12+iz|^{-2})+\\
&\quad+|\zeta(\tfrac12-iz+is)|^{2k}\min((H/X)^2,|\tfrac12-iz+is| ^{-2})
\end{align*}
and Corollary~\ref{corcon} and thanks to the fast decay of the Mellin transform $\hat g$, we can truncate the integral at $|s|\ll (\log X)^\eta$ at the cost of a $o(1)$ error. This implies in particular that we can remove the condition $|z-s|\in E_1\cup E_2\cup E_3$ at a negligible cost. We now consider the contribution from $|z|\in E_3$. By the above bound for $F$ we have that this range contributes
 \begin{align*}
&\ll \frac X{H}\frac1{(2\pi)^2 }\iint_{|s|\ll (\log X)^\eta,\atop |z|\in E_3}  |\hat g(1+is)\big(|\zeta(\tfrac12+iz)|^{2k}|\tfrac12+iz|^{-2}+|\zeta(\tfrac12-iz+is)^{2k}|\tfrac12-iz+is|^{-2}\big)\,dzds \\
&\ll (\log X)^{k^2-1}/Z.
\end{align*}
The same computation gives the bound $O((\log X)^{k^2-1}(\log \log X)^2/Z)$ for the range $|z|\in E_2$, $s<(\log X)^{-1}(\log\log X)^2$.
Also, since 
\begin{align*}
\partial_z F(z,s,u) &=   i \log(1+u)\big((1+u)^{\frac12-iz+is}-(1+u)^{\frac12+iz}\big)\\
&\ll \frac{H}X \min(1,|\tfrac12-iz+is| H/X+|\tfrac12+iz|H/X),
\end{align*}
then $\partial_z F(z,s,u) \ll H/X$ for $|z|\in E_2$. Thus, integrating by parts using Corollary~\ref{corcon} we than obtain that the contribution of $|z|\in E_2$, $(\log X)^{-1}(\log\log X)^2\leq |s|< (\log X)^\eta$, is 
\begin{align*}
&\ll ( \log X)^{k^2}\frac1{(2\pi)^2 }\iint_{\frac{Z X}{H} \leq |z| \leq \frac{Z X\log X}{H},\atop |s|\geq (\log X)^{-1}(\log\log X)^2} |\hat g(1+is)| (1+\min(1,|s|)\log X)^{-k^2/2}\frac{1}{|\frac12+iz|}\,dzds\\
&\ll ( \log X)^{k^2-1} (\log\log X)^{3-k^2}\ll ( \log X)^{k^2-1}/\log\log X.
\end{align*}
It remains to deal with $|z|\in E_1$. In this case $\partial_z F(z,s,u)\ll \frac{H^2}{X^2}|\frac12+iz|$ and thus integrating by parts once again we can bound this contribution by
\begin{align*}
&\ll ( \log X)^{k^2}\frac{H}{X}\frac1{(2\pi)^2 }\iint_{ s\ll (\log X)^\eta,\atop |z| \leq \frac{ X}{H Z}} |\hat g(1+is)| (1+\min(1,|s|)\log X)^{-k^2/2}\,dzds\\
&\ll ( \log X)^{k^2-1}/Z.
\end{align*}
This concludes the proof of~\eqref{bame} provided that $Z/(\log\log X)^2\to\infty$. We then have
\begin{align*}
\mathcal V_k(X,H)&=\frac1{HX}\int_{\R}g(x/X)\bigg|\frac1{2\pi  }\int_{\frac{X }{HZ}<|z|<\frac{X Z}{H}} \zeta^k(\tfrac12+iz)  \frac{(x+H)^{\frac12+iz}-x^{\frac12+iz}}{{\frac12+iz}}\,dz \bigg|^2\,dx+o(\log X)^{k^2-1})\\
&=\frac1{H}\frac1{(2\pi)^2 }\int_{\frac{X}{ZH}<|z|,|z-s|<\frac{ZX}{H}} \zeta^k(\tfrac12+iz)\zeta^k(\tfrac12-iz+is) J(z,s)\,dz ds+o((\log X)^{k^2}),
\end{align*}
where 
\begin{align*}
J(z,s)&:=\frac1X\int_{\R}g(x/X)\frac{(x+H)^{\frac12+iz}-x^{\frac12+iz}}{{\frac12+iz}}\frac{(x+H)^{\frac12-iz+is}-x^{\frac12-iz+is}}{{\frac12-iz+is}}\,dx\\
&=\frac1X\int_{\R}g(x/X)\frac{F(z,s,H/x)}{(\frac12+iz)(\frac12-iz+is)}x^{1+is}\,dx.
\end{align*}
Now, for $m\geq 0$ we have
\begin{align*}
\partial_x^m F(z,s,H/x)&\ll_m \frac{H^m}{x^{2m}}(1+|z|+|z-s|)^m,\\
\partial_x^m\partial_z F(z,s,H/x)&\ll_m  \frac{H^{m+1}}{x^{2m+1}}(1+|z|+|z-s|)^m. 
\end{align*}
Thus, integrating by parts $\ell\geq1$ times we have
\begin{align}
J(z,s)&\ll\frac {\frac{H^{\ell}}{X^{\ell-1}}(1+|z|+|z-s|)^\ell+X}{|1+is|^\ell |1+iz||1-iz+is|}\ll\frac {Z^\ell  X}{|1+is|^\ell |1+iz||1-iz+is|},\label{ab1}\\
\partial_z J(z,s)&\ll\frac {\frac{H^{\ell}}{X^{\ell-1}}(1+|z|+|z-s|)^\ell+X}{|1+is|^3|1+iz||1-iz+is|}\Big(\frac HX+\frac1{1+\min(|z|,|z-s|)}\Big)\\
&\ll\frac {Z^{\ell+1} H}{|1+is|^\ell|1+iz||1-iz+is|},\label{ab2}
\end{align}
%
for $\frac{X}{ZH}\ll |z|,|z-s|\ll \frac{ZX}{H}$. Thus, using
\begin{align*}
\frac{\zeta(\tfrac12+iz)^k\zeta(\tfrac12-iz+is)^k}{(\frac12+iz)(\frac12-iz+is)}&\ll |\zeta(\tfrac12+iz)|^{2k}|\tfrac12+iz|^{-2}+|\zeta(\tfrac12-iz+is)|^{2k}|\tfrac12-iz+is| ^{-2},
\end{align*}
and~\eqref{ab1}, we deduce by Corollary~\ref{corcon} that the contribution of $|s|\gg (\log X)^\eta$ to $\mathcal V_k(X,H)$ is $O(Z^{1/\eta+4}(\log X)^{k^2-2})$. For the contribution of $(\log X)^{-1+\eps}<|s|<( \log X)^\eta$, we integrate by parts in $z$. Using~\eqref{ab1}-\eqref{ab2} and Conjecture~\ref{shifted moments conjecture}, we have that this contribution is bounded by
\begin{align*}
&\ll Z^5 (\log X)^{k^2} \frac{H}{X}\frac1{(2\pi)^2 }\int_{\frac{X}{ZH}<|z|,|z-s|<\frac{ZX}{H},\atop (\log X)^{-1+\eps}<|s|<(\log X)^\eta}  (1+\min(1,|s|)\log X)^{-k^2/2} \,\frac{dz_1 ds}{(1+|s|)^3}\\
&\ll Z^6 (\log X)^{\frac{k^2}2}+ Z^6 (\log X)^{k^2-1-\eps}
\end{align*}
which is $o((\log X)^{k^2-1})$ if we take $Z:=(\log \log X)^4\ll (\log X)^{ \eps /7}.$ 

For the range $|s|<(\log X)^{-1+\eps}$, we use integration by parts, apply Conjecture~\ref{shifted moments conjecture}, and integrate by parts again (estimating the extremes both times) and obtain

 \begin{align*}
\mathcal V_k(X,H)&=-\frac {a_k}{H k!}\frac1{(2\pi i)^{2+k} }\int_{|s|<(\log X)^{-1+\eps}}\int_{\frac{X}{ZH}<|z|<\frac{ZX}{H}} \iint_{R^k}\bigg(\sum_{j=1}^kw_j +1\bigg) |z|^{\sum_{j=1}^kw_j}\\
&\quad\times\prod_{i=1}^kw_i^{-k}(w_i+s)^{-k}\prod_{i\neq j}(w_i-w_j)  \boldsymbol {dw} J(z,s)\,dz_1 ds+o(\log X)^{k^2-1})
\end{align*}
for $R$ the rectangle of vertexes $\pm 1/(\log X)\pm i /(\log X)^{1-2\eps}$. Now, on the domain of integration we have 
\begin{align*}
&\log |z| \sum_{j=1}^kw_j = \log (X/H)+O(\log Z/(\log X )^{-1}),\\
&J(z,s)=\frac1X\int_{\R}g(x/X)\frac{F(z,0,H/x)+O(ZH/(X\log X))}{|\frac12-iz|^2}x X^{s}\bigg(1+O\Big(\frac{\log Z}{(\log X)^{1-\eps}}\Big)\bigg)\,dx.
\end{align*}
It follows that we can replace $|z|^{\sum_{j=1}^kw_j}$ by $|X/H|^{\sum_{j=1}^kw_j}$, $\sum_{j=1}^kw_j +1$ by $1$, and $J(z,s)$ by $X^{s}J(z,0)$ at the cost of a negligible error. 
Thus,
 \begin{align*}
\mathcal V_k(X,H)&=-\frac {a_k}{H k!}\frac1{(2\pi i)^{2+k} }\int_{|s|<(\log X)^{-1+\eps}}\int_{\frac{X}{ZH}<|z|<\frac{ZX}{H}} \iint_{R^k} \Big(\frac{X}{H}\Big)^{s} |X/H|^{\sum_{j=1}^kw_j}\\
&\quad\times\prod_{i=1}^kw_i^{-k}(w_i+s)^{-k}\prod_{i\neq j}(w_i-w_j)  \boldsymbol {dw} X^sJ(z,0)\,dz ds+o(\log X)^{k^2-1}).
\end{align*}
Now
\begin{align*}
&\frac {i}{2\pi i}\int_{\frac{X}{ZH}<|z|<\frac{ZX}{H}} J(z,0)= \frac1{2\pi }\int_{\R}\frac1X\int_{\R}g(x/X)x\frac{F(z,0,H/x)}{|\frac12+iz|^2} \,dx dz+O\Big(\frac{H}{Z}\Big)\\
&\qquad= \frac1X\int_{\R}g(x/X)\frac1{2\pi i}\int_{\frac12-i\infty}^{\frac12+i\infty} \frac{2+H/x-(1+H/x)^{1-z}-(1+H/x)^{z}.}{z(1-z)} \, dz\, x\,dx+O\Big(\frac{H}{Z}\Big)\\
&\qquad=\int_\R g(x)\,dx\cdot H+O\Big(\frac{H}{Z}\Big),
\end{align*}
by the residue theorem. Recalling that $X/H=X^{1/\alpha}$, it follows that
 \begin{align*}
\mathcal V_k(X,H)&=\int_\R g(x)\,dx\cdot \frac {a_k}{ k!}\frac i{(2\pi i)^{1+k} }\int_{|s|<(\log X)^{-1+\eps}} \iint_{R^k} (X^{1/\alpha})^{\alpha s +\sum_{j=1}^kw_j}\\
&\quad\times\prod_{i=1}^kw_i^{-k}(w_i+s)^{-k}\prod_{i\neq j}(w_i-w_j)  \boldsymbol {dw}  ds+o(\log X)^{k^2-1}).
\end{align*}
We then obtain the claim result, by extending the integral over $s$, moving the line to $\Re(s)=\frac1{\log X}$, and proceeding in the same way as in the proof of Theorem~\ref{t1}.

\subsection{Proof of Theorem~\ref{tt3}}
%

First, assume $\alpha<k+\eps/2$. Writing
\begin{align*}
I_\chi(X):=\sum_{n\leq X}d_k(n)\chi(n)-\operatorname{Res}_{s=1}\Big(L(s,\chi)^{k}\frac{X^s}s\Big),
\end{align*}
by orthogonality we find
\begin{align}\label{wf}
\mathcal W_k(X;q)
&=\frac1X\sum_{1\leq a\leq q\atop (a,q)=1}\bigg|\frac1{\varphi(q)}\sum_{\chi\textnormal{ mod } q}\overline{\chi(a)}I_\chi(X)\Big)\bigg|^2
=\frac1{\varphi(q)}\sum_{\chi\textnormal{ mod } q}|X^{-1/2}I_{\chi}(X)|^2.
\end{align} 
We use~\cite[Lemma 3.12]{Tit} to write
\begin{align}\label{fma}
X^{-1/2} I_{\chi}(x)&=\frac1{2\pi i}\int_{c-i T}^{c+iT}L(\tfrac12+z,\chi)^k \frac{X^{z}}{\frac12+z}\,dz+O(X^{-\eps/2}),
 \end{align}
 with $c=1/\log X$ and $T=X^{1+\frac{\eps}{2k}}$.  Inserting this expression into~\eqref{wf} and applying Conjecture~\ref{shifted moments conjecture dirichlet}, we arrive to an expression almost identical to~\eqref{mainterm1}, with the only difference that now we have a factor of $\frac{-1}{(\frac12+z_1)(\frac12-z_2)}$ instead of $\frac1{z_1z_2}$. Then one concludes following the same steps as in  the proof of Theorem~\ref{t1}. Notice that in this case we obtain $\gamma_k$ rather its integral because in this case we have $\frac1{1+z}$ instead of $\frac1z$ in~\eqref{ifac}.

If $\alpha\geq k+\eps/2$, one moves the lines of integration in~\eqref{wf} to $\Re(z)=-\frac14$. Bounding trivially one obtains that $X^{-1/2}I_{\chi}(x)=o(1)$ and the result follows.

\section{Averages of secular coefficients}\label{rmt}
In the following proposition we show that the Random Matrix Theory analogue of $\mathcal I_k$ is asymptotically $N^{k^2}$ times the integral of $\gamma_k$, thus proving (unconditionally) the RMT analogue of Theorem~\ref{scd}. It also gives an alternative (simpler) complex analytic proof of~\eqref{ffik} (cf. in particular,~\cite[Theorem 1.6 and Section~4.4]{KRRR}).

\begin{proposition}
For $n,m \in\N$ we have
\begin{align}
\tilde I_k(n,N)&:=\int_{U(N)}\bigg|\sum_{j_1+\cdots+j_k \leq m\atop 0\leq j_1,\dots,j_k\leq N}\Sc_{j_1}(A)\cdots \Sc_{j_k}(A)\bigg|^2\,dA,\notag\\
&=N^{k^2}\int_0^{n/N}\gamma_k(x)\,dx+O_k(N^{k^2-1})\label{if1}\\
&=\frac{N^{k^2}}{k^2!}\sum_{0\leq r<n/N}\binom{k}r P_{r,k}(n/N)+O_k(N^{k^2-1}).\label{if2}
\end{align}
\end{proposition}
\begin{proof}
For $\eta<1$ we have
\begin{align*}
 I_k(n,N):=\int_{U(N)}\int_{|z_1|=\eta}\frac1{(2\pi i)^2}\int_{|z_2|=\eta}\Lambda_A(z_1)^k\Lambda_{A^*}(z_2)^k\frac{dz_1 dz_2}{(z_1z_2)^{n+1}}dA,
\end{align*}
where $A^*$ is the conjugate transpose of $A$. Now, we have
\begin{align*}
\int_{U(N)}\Lambda_A(z_1)^k\Lambda_{A^*}(z_2)^kdA=F_N(z_1z_2)
\end{align*}
for some polynomial $F_N$ of degree $kN$ (the fact that $F_N$ depends on the product $z_1z_2$ only follows from the invariance of Haar measure of $U(N)$ under multiplication by unit scalars). Making the change of variables $z_1=s/z_2$ we then find that 
\begin{align*}
 I_k(n,N)&=\frac1{(2\pi i)^2}\int_{|s|=\eta^2}\int_{|z|=\eta}F_N(s)\frac{ds dz}{zs^{n+1}}=\frac1{(2\pi i)^2}\int_{|s|=\eta^2}F_N(s)\frac{ds}{s^{n+1}}.
\end{align*}
Now, 
\begin{align*}
\sum_{j=1}^n\frac1{s^{n+1}}=\frac{1}{s^{n+1}}\frac{s^{n+1}-1}{s-1}=\frac1{s-1}-\frac{1}{s^{n+1}(s-1)}
\end{align*}
and so
\begin{align*}
\sum_{j=1}^n I_k(n,N)&=\frac1{2\pi i}\int_{|s|=\eta^2}F_N(s)\pr{\frac1{s-1}-\frac{1}{s^{n+1}(s-1)}}ds\\
&=-\frac1{2\pi i}\int_{|s|=\eta^2}\frac{F_N(s)}{s^{n+1}(s-1)}ds
\end{align*}
by the residue theorem. By~\eqref{ffik} we also have 
\begin{align*}
\sum_{j=1}^n I_k(n,N)=N^{k^2}\int_0^{n/N}\gamma_k(x)\,dx+O_k(N^{k^2-1})
\end{align*}
and so
\begin{align*}
\int_0^{n/N}\gamma_k(x)\,dx=-\frac1{2\pi i}\int_{|s|=\eta^2}\frac{N^{-k^2}F_N(s)}{s^{n+1}(s-1)}ds+O_k(N^{-1}).
\end{align*}
On the other hand, proceeding in the same way, we have
\begin{align*}
\tilde I_k(n,N)&=\sum_{m_1,m_2\leq n}\int_{U(N)}\frac1{(2\pi i )^2}\int_{|z_1|=\eta}\int_{|z_2|=\eta}\Lambda_A(z_1)^k\Lambda_{A^*}(z_2)^k\frac{dz_1 dz_2}{z_1^{m_1+1}z_2^{m_2+1}}dA\\
&=\int_{U(N)}\frac1{(2\pi i )^2}\frac1{(2\pi i )^2}\int_{|z_1|=\eta}\int_{|z_2|=\eta}\frac{\Lambda_A(z_1)^k\Lambda_{A^*}(z_2)^k}{z_1^{n+1}(z_1-1)}\frac{1}{z_2^{n+1}(z_2-1)}\, dz_1 dz_2dA\\
& =\frac1{2\pi i}\int_{|s|=\eta^2}\int_{|z|=\eta}\frac{F_N(s)}{s^{n+1}(s-z)(z-1)}\, ds dz\\
& =-\frac1{2\pi  i}\int_{|s|=\eta^2}\frac{F_N(s)}{s^{n+1}(s-1)}\, ds 
\end{align*}
and~\eqref{if1} follows. 

It remains to show~\eqref{if2}. By~\cite{CFKRS} (see also~\cite[Lemma 3]{CRS}) if $\frac12<z<1$ we have
\begin{align*}
F_N(z)&=\frac1{k!(2\pi i)^k}\iint_{|w_i|=1}e^{N\sum_{i=1}^kw_i}\prod_{i\neq j,\atop 1\leq i,j\leq k}(1-e^{-w_i+w_j})\prod_{i=1}^k(1-e^{-w_i})^{-k}(1-ze^{-w_i})^{-k}\,\boldsymbol{dw}.
\end{align*}
We assume also $z\leq1-1/N$ and we express each integral over $w_i$ as a sum of two integrals along circles of radius $\frac1{2N}$ around $0$ and $\log(z)$. We then obtain
\begin{align*}
F_N(z)=\sum_{r=0}^k\binom{k}r F_{N,r},\qquad
\tilde I_k(n,N)=-\sum_{r=0}^k\binom{k}r \frac1{2\pi  i}\int_{|s|=\eta^2}\frac{F_{N,r}(z)}{z^{n+1}(z-1)}\, dz,
\end{align*}
where
\begin{align*}
F_{N,r}(z)&=\frac1{(2\pi i)^k k!}\iint_{|w_i|=\frac1{2N}}z^{rN}e^{N\sum_{i=1}^kw_i}\prod_{i\neq j,\atop 1\leq i,j\leq k}(1-z^{-\delta_{i,r}+\delta_{j,r}}e^{-w_i+w_j})\\[-0.6em]
&\hspace{10em}\times\prod_{i=1}^k(1-z^{-\delta_{i,r}}e^{-w_i})^{-k}(1-z^{1-\delta_{i,r}}e^{-w_i})^{-k}\,\boldsymbol{dw}\\
&=\frac1{(2\pi i)^k k!}\iint_{|w_i|=\frac1{2N}}z^{rN+r^2+r}e^{N\sum_{i=1}^kw_i}\prod_{i\neq j,\atop 1\leq i,j\leq k}(z^{1-\delta'_{i,j,r}}-z^{\delta'_{i,j,r}}e^{-w_i+w_j})\\[-0.6em]
&\hspace{10em}\times\prod_{i=1}^k(z^{\delta_{i,r}}-e^{-w_i})^{-k}(1-z^{1-\delta_{i,r}}e^{-w_i})^{-k}\,\boldsymbol{dw},
\end{align*}
with $\delta_{i,r}:=1$ if $1\leq i\leq r$ and $\delta_{i,r}=0$ otherwise and $\delta'_{i,j,r}:=\max(0,-\delta_i+\delta_j)$.
 For $0<|z|<1-1/N$, $|w_i|=\frac1{2N}$ we have $|z e^{w_i}|<1$ and so the above expression holds by analytic continuation also for $z\in\C$, $|z|<1/N$. Inserting the above expression for $F_{N,r}$ and exchanging the order of integration we obtain, by Cauchy's theorem, that
\begin{align*}
\frac1{2\pi i}\int_{|z|=\delta^2}F_{N,r}(z)\frac{dz}{z^{n+1}(z-1)}=0,
\end{align*}
if $ N r -n -1 + r^2+r \geq0$ and so in particular if $N r\geq n+1$. If $Nr \leq n$ we again exchange the order of integration and move the circuit of integration to $|z|=2$, collecting residues at $z=e^{\pm w_i}$ and $z=0$. We express the contribution of the residues as an integral along a circle of radius $\frac1{2N}$ and center $1$. We then obtain 
\begin{align*}
\frac1{2\pi i}\int_{|z|=\eta^2}\frac{F_{N,r}(z) dz}{z^{n+1}(z-1)}&=\frac1{(2\pi i)^{k+1}}\int_{|z-1|=\frac1N}\iint_{|w_i|=\frac1{2N}}z^{rN+r^2+r-n-1}e^{N\sum_{i=1}^kw_i}\\[-0.6em]
&\hspace{-5em}\times\prod_{i\neq j,\atop 1\leq i,j\leq k}(z^{1-\delta'_{i,j,r}}-z^{\delta'_{i,j,r}}e^{-w_i+w_j})\prod_{i=1}^k(z^{\delta_{i,r}}-e^{-w_i})^{-k}(1-z^{1-\delta_{i,r}}e^{-w_i})^{-k}\,\boldsymbol{dw}+E,
\end{align*}
where $E$ is the contribution of the integral on the circuit $|z|=2$. If $ N r -n -1 + r^2+r \leq -2$ we can move this circuit to infinity and obtain that $E=0$. Otherwise, if $ n  - r^2-r  \leq N r \leq n$, then exchanging the order of integration  once again, we have that inside the paths of integration the integrand have a pole at $w_1=\cdots=w_k=0$ only. This pole is of order $k^2-r(r-1)-(k-r-1)(k-r-2)$ (a pole of order $k^2$ from the second product and a simple zero each time that $\delta'_{i,j,r}=0$); since we have $k$ variables we then obtain $E=O(N^{-2 + 3 k - 2 r + 2 k r - 2 r^2})$, where we also used that  in this case $z^{rN+r^2+r-n-1}=2^{O_k(1)}$. Since $\min_{r\in\R}(-2 + 3 k - 2 r + 2 k r - 2 r^2)\leq k^2-1$ it follows that $E=O(N^{k^2-1})$.

Finally, the integral along the circle $|z-1|=1/(2N)$ can be reduced, up to $O(N^{k^2-1})$ error, to an integral analogous to~\eqref{ifac} by using the approximations
\begin{align*}
& z^{rN-n+O_k(1)}=e^{s(rN-n)}(1+O_k(N^{-1})) && z^{1-\delta}-z^{\delta}e^{-w_i+w_j}=(1-2\delta)s+w_i-w_j+O(N^{-2})\\
&1-z^{\delta}e^{-w_i}=w_i-\delta s+O(N^{-2}) && z^{\delta}-e^{-w_i}=w_i+\delta s+O(N^{-2})
\end{align*}
for $s:=z-1=O(1/N)$ and $\delta\in\{0,1\}$. One then obtains~\eqref{if2} in the same way as above.
\end{proof}

\begin{corollary}\label{cff}
For $\alpha\geq0$ we have
$$
\frac1{k^2!}\sum_{0\leq r<\alpha}\binom{k}r P_{r,k}(\alpha)=\int_0^\alpha\gamma_k(x)\,dx.
$$
\end{corollary}

\begin{proposition}\label{p1}
Let $k\geq1$ and $0\leq r\leq k$. Then $P_{r,k}(\alpha)$ has a zero of order at least $(k-r)^2 + r^2$ at $\alpha=r$.
\end{proposition}
\begin{proof}
This follows from the previous Corollary and from the smoothness properties of $\gamma_k$. It can also be deduced directly from the definition of $P_{r,k}$. Indeed, when $\alpha=r$ the integrand in the definition of $P_{r,k}$ has poles of order $k$ at $w_i=-1$ if $i\leq r$ and at $w_i=0$ if  $r< w_i\leq k$. Then the Vandermonde determinant gives a zero of multiplicity $r(r-1)$ when $w_i=-1$, $1\leq i\leq r$, and a zero of multiplicity $(k-r)(k-r-1)$ when $w_i=0$, $r<i\leq k$, whereas $(\sum_{j=1}^kw_j+\alpha)^{k^2}$ gives a zero of order $k^2$. Then, considering we are integrating over $k$ variables, we have that  $P_{r,k}$ has a zero of order $-rk-(k-r)k+r(r-1)+(k-r)(k-r-1)+k^2+k=(k-r)^2 + r^2$.
\end{proof}


\begin{thebibliography}{0}
\bibitem{BGR}
Basor, E.; Ge, F.; Rubinstein, M. O. \emph{Some multidimensional integrals in number theory and connections with the Painlev\'e V equation}. J. Math. Phys. 59 (2018), no. 9, 091404, 14 pp.
\bibitem{Bet}
Bettin, S. \emph{The second moment of the Riemann zeta function with unbounded shifts}. Int. J. Number Theory 6 (2010), no. 8, 1933--1944. 
\bibitem{BBLR}
Bettin, S.; Bui, H. M.; Li, X.; Radziwi\l\l, M. \emph{A quadratic divisor problem and moments of the Riemann zeta-function}. Preprint, arXiv:1609.02539.
\bibitem{Blo}
Blomer, V.: \emph{The average value of divisor sums in arithmetic progressions}. Q. J. Math. 59, 275--286 (2008).
\bibitem{CFKRS}
Conrey, J. B.; Farmer, D. W.; Keating, J. P.; Rubinstein, M. O.; Snaith, N. C. { \it Autocorrelation of random matrix polynomials}. Comm. Math. Phys. 237 (2003), no. 3, 365-395.
\bibitem{CFKRS2}
Conrey, J. B.; Farmer, D. W.; Keating, J. P.; Rubinstein, M. O.; Snaith, N. C. { \it Integral moments of $L$-functions} Proc. London Math. Society, 91 (2005), 33-104.
\bibitem{CGh}
Conrey, J. B.; Ghosh, A. \emph{A conjecture for the sixth power moment of the Riemann zeta-function}.  Internat. Math. Res. Notices,  1998, 15, 775--780.
\bibitem{CGo}
Conrey, J. B.; Gonek, S.M. \emph{High moments of the Riemann zeta-function}.  Duke Math. J.  107  (2001),  no. 3, 577--604.
\bibitem{CIS}
Conrey, J. B.; Iwaniec, H.; Soundararajan, K. \emph{The sixth power moment of Dirichlet $L$-functions}. Geom. Funct. Anal. 22 (2012), no. 5, 1257--1288.
\bibitem{CK1}
Conrey, J. B.; Keating, J. P. \emph{Moments of zeta and correlations of divisor-sums: I}.  Philos. Trans. Roy. Soc. A 373 (2015), no. 2040, 20140313, 11 pp.  
\bibitem{CK2}
Conrey, J. B.; Keating, J. P. \emph{Moments of zeta and correlations of divisor-sums: II}. Advances in the theory of numbers, 75--85, Fields Inst. Commun., 77, Fields Inst. Res. Math. Sci., Toronto, ON, 2015.
\bibitem{CK3}
Conrey, J. B.; Keating, J. P. \emph{Moments of zeta and correlations of divisor-sums: III}.  Indag. Math. (N.S.) 26 (2015), no. 5, 736--747. 
\bibitem{CK4}
Conrey, J. B.; Keating, J. P. \emph{Moments of zeta and correlations of divisor-sums: IV}.  Res. Number Theory 2 (2016), Art. 24, 24 pp.
\bibitem{CK5}
Conrey, J. B.; Keating, J. P. \emph{Moments of zeta and correlations of divisor-sums: V}. Proc. Lond. Math. Soc. (3) 118 (2019), no. 4, 729--752. 
\bibitem{CR}
Conrey, J. B.; Rodgers, B. \emph{Averages of quadratic twists of long Dirichlet polynomials}. Preprint.
\bibitem{CRS}
Conrey, J. B.; Rubinstein, M. O.; Snaith, N. C. \emph{Moments of the derivative of characteristic polynomials with an application to the Riemann zeta function}. Comm. Math. Phys. 267 (2006), no. 3, 611--629.
\bibitem{CS}
Conrey J. B. and Snaith N. C. \emph{Applications of the $L$-functions ratios conjectures.} Proc. Lond. Math. Soc. 94 (2007), no. 3, 594--646.
\bibitem{CSa}
Coppola, G., Salerno, S.: On the symmetry of the divisor function in almost all short intervals. Acta Arith. 113, 189--201 (2004).
\bibitem{dF}
de la Bret\`eche, R.; Fiorilli, D. \emph{Major arcs and moments of arithmetical sequences}. Amer. J. Math. 142, no. 1, Feb. 2020,  45--77.
\bibitem{GL}
Ge, F.; Liu, G. \emph{A combinatorial identity and the finite dual of infinite dihedral group algebra}. Preprint, arXiv:2005.01410.
\bibitem{GH}
Goldfeld, D.; Hoffstein, J. \emph{Eisenstein series of $\frac12$-integral weight and the mean value of real Dirichlet $L$-series}. Invent. Math. 80 (1985), no. 2, 185-208.
\bibitem{GGM}
Goldston, D. A.; Gonek, S.M; Montgomery, H. L. \emph{Mean values of the logarithmic derivative of the Riemann zeta-function with applications to primes in short intervals}. J. Reine Angew. Math. 537 (2001), 105--126. 
\bibitem{GR}
Gorodetsky, O.; Rodgers, B. \emph{The variance of the number of sums of two squares in $\mathbb F_q[T]$ in short intervals}. Preprint, arXiv:1810.06002.
\bibitem{HKR}
Hall, C.; Keating, J. P.; Roditty-Gershon, E. \emph{Variance of arithmetic sums and L-functions in $\F_q[t]$}. Algebra Number Theory 13 (2019), no. 1, 19--92. 
\bibitem{HL}
Hardy, G. H.; Littlewood, J. E. \emph{Contributions to the theory of the Riemann zeta-function and the theory of the distribution of primes}. Acta Mathematica 41 (1918), 119 - 196.
\bibitem{HS}
Harper, A. J.; Soundararajan, K. \emph{Lower bounds for the variance of sequences in arithmetic progressions: primes and divisor functions}. Q. J. Math. 68 (2017), no. 1, 97--123.
\bibitem{HY}
Hughes, C. P.; Young, M. P. \emph{The twisted fourth moment of the Riemann zeta function}. J. Reine Angew. Math. 641 (2010), 203--236
\bibitem{Ing}
Ingham, A. E. \emph{Mean-values theorems in the theory of the Riemann zeta-function}. Proc. Lond. Math. Soc., 27, 1926, p. 273--300.
\bibitem{Ivi}
Ivi\'c, A. \emph{On the mean square of the divisor function in short intervals}. J. Théor. Nombres Bordeaux 21 (2009), no. 2, 251--261.
\bibitem{Ivi2}
Ivi\'c, A. \emph{On the divisor function and the Riemann zeta-function in short intervals}. Ramanujan J. 19 (2009), no. 2, 207--224. 
\bibitem{Jut}
Jutila, M. \emph{On the divisor problem for short intervals}. Studies in honour of Arto Kustaa Salomaa on the occasion of his fiftieth birthday. Ann. Univ. Turku. Ser. A I No. 186 (1984), 23--30.
\bibitem{KRRR}
Keating, J. P.; Rodgers, B.; Roditty-Gershon, E.; Rudnick, Z. \emph{Sums of divisor functions in $\mathbb F_q[t]$ and matrix integrals}. Math. Z. 288 (2018), no. 1-2, 167--198.
\bibitem{KS}
Keating, J. P.; Snaith, N. C. \emph{Random matrix theory and $\zeta\pr{\frac12 + it} $}. Comm. in Math. Phys. 214 (2000), 57Ð89.
\bibitem{KR}
Kowalski, E., Ricotta, G. \emph{Fourier coefficients of $\textnormal{GL(N)}$ automorphic forms in arithmetic progressions}. Geom. Funct. Anal. 24, 1229--1297.
\bibitem{LZ}
Lau, Y.K.; Zhao, L. \emph{On a variance of Hecke eigenvalues in arithmetic progressions}. J. Number Theory 132(5), 869--887 (2012).
\bibitem{Les}
Lester, S. \emph{On the variance of sums of divisor functions in short intervals}. Proc. Amer. Math. Soc. 144 (2016), no. 12, 5015--5027.
\bibitem{MR1}
Matom\"aki, K.; Radziwi\l\l, M. \emph{Multiplicative functions in short intervals}. Ann. of Math. (2) 183 (2016), no. 3, 1015--1056. 
\bibitem{MR2}
Matom\"aki, K.; Radziwi\l\l, M. \emph{A note on the Liouville function in short intervals}. Preprint, arXiv:1502.02374. 
\bibitem{MT}
Milinovich, M. B.; Turnage-Butterbaugh, C. L. \emph{Moments of products of automorphic L-functions}. J. Number Theory 139 (2014), 175?204.
\bibitem{MV}
Montgomery, H.L.; Vaughan, R.C. \emph{Hilbert's inequality}. J. London Math. Soc. (2) 8 (1974), 73--82.
\bibitem{Mot}
Motohashi, Y. \emph{On the distribution of the divisor function in arithmetic progressions}. Acta Arith. 22, 175--199 (1973).
\bibitem{RS}
Rodgers, B.; Soundararajan, K. \emph{The variance of divisor sums in arithmetic progressions}. Forum Math. 30 (2018), no. 2, 269--293.
\bibitem{SV}
Saffari B.; Vaughan, R. C. \emph{On the fractional parts of $x/n$ and related sequences. II}, Ann. Inst. Fourier (Grenoble) 27 (1977), v, 1--30.
\bibitem{Shi}
Shimomura, S. \emph{Shifted fourth moment of the Riemann zeta-function}. Acta Math. Hungar. 137 (2012), no. 1-2, 104--129.
\bibitem{Tit}
Titchmarsh, E.C. {\it The Theory of the Riemann Zeta-function}. Oxford Science Publications, second edition, 1986.

\end{thebibliography}
 \end{document}